\documentclass[review,3p,12pt]{elsarticle}
\usepackage{lineno,hyperref}
\usepackage{lipsum}
\journal{ }

\def\R{{\mathbb R}}

\def\dee{{\rm d}}
\def\e{{\rm e}}
\def\:{{\colon}}
\def\loc{{\rm loc}}

\usepackage{graphicx,epsfig,psfrag}
\usepackage{amssymb}
\usepackage{color}
\usepackage{amsthm}
\usepackage{mathtools}
\usepackage{mathrsfs}
\usepackage{calrsfs}
\usepackage{mathscinet}
\usepackage{hyperref}
\usepackage{amsmath}
\usepackage{enumerate}
\numberwithin{equation}{section}
\newtheorem{theorem}{Theorem}[section]
\newtheorem{corollary}[theorem]{Corollary}
\newtheorem{lemma}[theorem]{Lemma}
\newtheorem{proposition}[theorem]{Proposition}

\newdefinition{remark}{Remark}[section]
\newdefinition{definition}{Definition}[section]
\newdefinition{example}{Example}[section]

\newcommand{\bel}{\begin{flalign}}
\newcommand{\eel}{\end{flalign}}
\newcommand{\be}{\begin{equation}}
\newcommand{\ee}{\end{equation}}
\newcommand{\ba}{\begin{eqnarray}}
\newcommand{\bs}{\begin{eqnarray*}}
\newcommand{\ea}{\end{eqnarray}}
\newcommand{\es}{\end{eqnarray*}}
\newcommand{\bi}{\begin{itemize}}
\newcommand{\ei}{\end{itemize}}
\newcommand{\A}{{\alpha}}

\newcommand{\Om}{{\Omega   }}
\newcommand{\Lap}{\Delta}
\newcommand{\B}{{\beta}}

\newcommand{\ITB}{\int_0^{T}}
\newcommand{\It}{\int_0^t}
\newcommand{\Is}{\int_{\R^n }}

\newcommand{\F}{{\mathscr F}}
\newcommand{\upperu}{\overline{u}}
\newcommand{\loweru}{\underline{u}}

\allowdisplaybreaks
\begin{document}

\begin{frontmatter}

\author[RL]{R.\ Laister\corref{cor}}
\ead{Robert.Laister@uwe.ac.uk}
\address[RL]{Department of Engineering Design and Mathematics, \\ University of the West of England, Bristol BS16 1QY, UK.}
\author[warsaw]{M.\ Sier{\.z}\polhk{e}ga}
\ead{M.Sierzega@mimuw.edu.pl}
\address[warsaw]{Faculty of Mathematics, Informatics and Mechanics,  University of Warsaw,\\  Banacha 2, 02-097 Warsaw, Poland.}
%
\cortext[cor]{Corresponding author}

\title{Well-posedness of  Semilinear Heat Equations in $L^1$}

\begin{abstract}
The problem of obtaining necessary and sufficient conditions for local existence of non-negative solutions in Lebesgue spaces for  semilinear heat equations having monotonically increasing source term $f$ has only recently been resolved  (Laister  et al. (2016)). There, for the more difficult case of initial data in $L^1$, a necessary and sufficient integral condition on $f$ emerged. Here, subject to this integral condition, we consider other fundamental properties of   solutions with $L^1$ initial data of indefinite sign, namely: uniqueness, regularity, continuous dependence and comparison. We also establish sufficient conditions for the global-in-time continuation of solutions for small initial data in  $L^1$.

\end{abstract}

\end{frontmatter}


\section{Introduction}
In this paper we address fundamental questions concerning the well-posedness of semilinear heat equations of the form
\begin{equation}\label{nhe}
u_t=\Lap u+f(u),\qquad u(0)=\phi ,
\end{equation}
where  $f\colon\R\to\R$ is locally Lipschitz continuous and non-decreasing, $f(0)=0$ and the initial condition $\phi$ is taken in the space $L^1(\R^n )$.
In a previous work \cite{LRSV} we  established sharp results on the local existence of non-negative solutions when  $f\colon [0,\infty)\to[0,\infty)$ and   $\phi\in L^1_+(\R^n )$. Here we consider issues of existence and  uniqueness for  initial data of indefinite sign, thus providing a more comprehensive  account of the well-posedness problem for  \eqref{nhe}. Furthermore we obtain results on the regularity, continuous dependence, comparison and   global existence of solutions.


The special case of the homogeneous power law nonlinearity $f(u)=|u|^{p-1}u$ in (\ref{nhe}), commonly referred to as the Fujita equation, has attracted much attention. In fact, the original spark of interest   in this restricted setting \cite{Fuj66} has inspired the bulk of subsequent developments. The interested reader is invited to consult \cite{QS19} and the extensive list of references therein for a detailed account of the state of the art for this  problem.

A particularly important avenue of research concerns local well-posedness for {\em singular} (unbounded) initial data and {\em power-like} nonlinearities satisfying
\be
 |f(u)-f(v)|\leq C\left(1+|u|^{p-1}+|v|^{p-1}\right)|u-v|,\qquad p>1.\label{eq:power}
 \ee
Sufficient conditions for local existence in Lebesgue spaces for classes of nonlinearities of this type   were established via a contraction mapping argument in \cite{W79,W80} (see also \cite{BC}). There the interplay between  the power-like nonlinearity  and Lebesgue  norms produced a clear-cut characterisation of local existence in $L^q$  with respect to a particular  exponent $q^\ast=\frac{n(p-1)}{2}$.  Thus  when $q\ge 1$ and $q>q^\ast$, or $q=q^\ast>1$,  then for every $\phi\in L^q$ there exists a $T_\phi >0$ and a solution in
 $C([0,T_{\phi}], L^q)\cap L^\infty_{\loc}((0,T_{\phi}),L^\infty)$, i.e. a   {\em classical $L^q$-solution}. The special case of the pure power law $f(u)=|u|^{p-1}u$ shows that $q^\ast$ is indeed a critical value as the problem can exhibit non-existence of non-negative solutions  whenever $1\le q<q^\ast$ \cite[Theorem 5]{W80}. Lastly, the end-point case $q=q^\ast =1$ differs from the critical regime for $q_\ast>1$ in that
 there are non-negative data for which no non-negative solution may be defined  \cite[Theorem 11]{BC} and \cite{CZ}.


Uniqueness results were also established in \cite{W80}, but in a restricted  class of functions where solutions satisfy a certain growth
 bound as $t\to 0$ {\em a priori}.  This growth restriction was removed in  \cite{BC}, providing a stronger uniqueness result in the space
 of classical $L^q$-solutions.  Hitherto these were the best results available regarding local well-posedness of (\ref{nhe}) in $L^q$.


For non-power-like nonlinearities  this tight correspondence between the source term growth and integrability of the initial data generally fails and two problems arise:
\begin{enumerate}[(A)]
\item given $f$, characterise the set $X$ of initial data for which \eqref{nhe} has a solution;
\item given the set $X$ of initial data, characterise those sources $f$ for which \eqref{nhe} has a solution  for every initial datum in $X$. 
\end{enumerate}
These two problems are of course subject to a variety of further ramifications,  e.g.,we need to specify the solution concept and then we may insist that $X$ is a Banach space say, or that $f$ is a smooth function, etc. 

One definitive result in direction (A), concerning non-negative solutions of the Fujita equation on bounded domains with Dirichlet boundary conditions, was given in \cite{BP}. With regards to problem (B), in  \cite{LRSV}  we gave a necessary and sufficient condition for a non-negative, {\em non-decreasing}, continuous function $f$ to give rise to a local solution bounded in $L^q_+(\R^n )$ for all  initial data in $L^q_+(\R^n )$. 
Specifically, for $q>1$, and assuming $f$ to be Lipschitz continuous at zero, local existence holds for all $\phi\in L^q_+(\R^n )$   if and only if
\bs\label{eq:introLq}
\limsup_{u\to\infty}\frac{f(u)}{u^{1+2 q/n}}<\infty .
\es
Morally, this result means that the power-like nonlinearities of (\ref{eq:power}) are essentially the whole story  with regards to   local existence in $L_{+}^q(\R^n)$ when $q>1$, i.e. it is precisely those nonlinearities which are majorised by $u^{1+2 q/n}$ which are compatible with Lebesgue spaces of non-negative initial data. 
However,  power-like nonlinearities   certainly do not tell the whole story in the  more delicate  case $q=1$.  For then the condition  for existence 
\be
|f(u)-f(v)|\leq C\left(1+|u|^{p-1}+|v|^{p-1}\right)|u-v|,\qquad p<1+2/n\label{eq:IntroLq}
\ee
in \cite{BC,W80} is sufficient  but not necessary. In fact the optimal (`critical') condition obtained in  \cite{LRSV} for the local existence  property in $L^1_+(\R^n )$  reads 
\begin{equation}\label{eq:withFs}
 \int_1^\infty s^{-(1+2 /n)}\ell_+ (s)\,\dee s<\infty,
\end{equation}
where
\bs
\ell_+ (s)=\sup_{0< t\le s}\frac{f(t)}{t}.
\label{eq:ellplusintro}
\es
(Note that $\ell_+ (s)$ was denoted by $F(s)$ in  \cite{LRSV}.) 
Thus there are nonlinearities $f$ which satisfy the integral condition (\ref{eq:withFs}) but which do not satisfy (\ref{eq:IntroLq}).
See  \cite[Section 4.4]{LRSV} for such an example in the bounded domain case, but which also applies to the whole space $\R^n $. We see  again that $L^1$ differs qualitatively from higher Lebesgue spaces and concepts like `critical nonlinearity' require careful treatment.

It is worth mentioning that the contraction mapping arguments in \cite{BC, W79,W80} do not require monotonicity of $f$. Whilst we obtain stronger results in terms of the growth of $f$ at infinity, our methods are reliant upon  monotonicity.  For some alternative, related results see also the  recent developments  in  \cite{FI}. 


It is natural then to ask whether the gap between the existence conditions (\ref{eq:IntroLq}) and \eqref{eq:withFs} might lead to improvements  of 
 current uniqueness results  when $q=1$, namely \cite[Theorem 1]{BC}.  Indeed, we show in Theorem~\ref{thm:unique} that this is  the case. Whereas uniqueness of classical $L^1$-solutions is  known to hold   under   (\ref{eq:IntroLq}),
we are able to weaken this condition accordingly; see hypothesis   {\bf  (I2)} in Section~\ref{sec:unique}. There we adapt some of the methods of 
\cite{BC} but without recourse to the contraction mapping theorem. It should be noted however,  that our choice of working with classical solutions (for $t>0$) with initial data in $L^1(\R^n)$ is instrumental in uniqueness considerations;  non-uniqueness can result otherwise \cite{MT,NS}.
Furthermore our monotone methods yield additional   results on continuous dependence and comparison in this more general setting,  see Theorem~\ref{thm:comp}, which  extend further those of \cite{BC}. 

We go on to consider some special classes of source terms and initial data. For example, when $f$ is  convex on $[0,\infty )$ and odd  
 then  the integral condition for existence (namely {\bf  (I1)} of Section~\ref{sec:exist}, similar in spirit to that of (\ref{eq:withFs})) is  equivalent to the one for both existence {\em and} uniqueness  (namely {\bf  (I2)}) - see Corollary~\ref{cor:convexI1I2}. Specialising further to the case of non-negative ititial data  then leads  to a sharp (necessary and sufficient) result, Corollary~\ref{cor:summary}, regarding local well-posedness of (\ref{nhe}) within the class of convex source terms on $[0,\infty )$  and initial data in  $L^1_+(\R^n)$.

In the context of   global solutions it has long been known  for convex sources $f$ satisfying the ODE blow-up criterion
$\int^\infty\dee s/f(s)<\infty$,  that blow-up in (\ref{nhe}) can occur for {large} initial data and that small initial data
can evolve into global-in-time solutions, see \cite{Kaplan}.
In the seminal paper \cite{Fuj66}  Fujita obtained an important result for the special case of $f(u)=u^p$ on the whole space, whereby  the critical exponent $p_F:=1+{2}/{n}$ separated two regimes: for $1<p<p_F$ every non-negative, non-trivial classical solution  blows up in finite time, whereas for $p> p_F$ it is possible to find {small} initial conditions evolving into global-in-time solutions. The so-called `critical case' $p=p_F$ was later shown to be in the blow-up regime \cite{Hay,Sug}.  Notions of smallness abound, but Fujita's original approach involved data bounded above by a small Gaussian.
Subsequent developments
defined smallness via a norm in an appropriate Banach space. The homogeneous power  nonlinearity $f(u)=|u|^{p-1}u$ is naturally compatible with the Lebesgue space structure and in \cite{W81} Weissler succeeded in replacing the small Gaussian bound with a small critical Lebesgue norm $\|\cdot \|_{{p_F}}$. In this paper too we establish some  global existence results for initial data with sufficiently small $L^1$-norm, see Theorems~\ref{thm:global} and \ref{thm:globalL1}.


Critical and supercritical phenomena appear often in semilinear heat equations   when the averaging action of  diffusion and the amplifying action of the nonlinear term exert a comparable infulence on the behaviour of the solution. The character of this inteplay may lead to subtle effects in the dynamics of singularity formation \cite{FHV}, existence of nonsmooth solutions \cite{PY} and asymptotic behaviour of global solutions \cite{FKWY}.  In the context of power law  source terms, a considerable body of work on these topics has been 
 developed over the past three decades  (see \cite{QS19} for an overview). In this work we have investigated the critical balance with regards to the question of local existence of solutions for source terms in a more general class of functions.  Whether similar generalisations are possible in the blow-up theory and threshold behaviour of global solutions remains to be seen
 and poses interesting research challenges for the future.


\section{Well-Posedness in $L^1$}

\subsection{Preliminaries}

Throughout we use $p_F$ to denote the Fujita exponent  $\displaystyle{p_F= 1+{2}/{n}}$.  We write $\| \cdot \|_q$ for the norm in $L^q(\R^n  )$ and denote by $\left\{ S(t)\right\}_{t\ge 0}$  the heat semigroup on $L^q(\R^n  )$ ($q\ge 1$) generated by $-\Lap$ on  $\R^n$ with  the explicit representation formula
 \bs
 [S(t)\phi](x)=\Is G(x-y,t)\phi (y)\, \dee  y,\qquad \phi\in L^q(\R^n  ),\label{eq:heatsg}
 \es
where $G$ is the Gaussian heat kernel
\bs
G(x,t)=(4\pi t)^{-n/2}\e^{-|x|^2/4t}.\label{eq:gauss}
\es
It is well-known  that $S(t)$ is a $C_0$-semigroup for all  $q\ge 1$ finite.

 As is commonplace for semilinear problems we will  study solutions of (\ref{nhe}) via the variation of constants formula
\begin{equation}
u(t)={\mathscr F}(u;\phi)\coloneqq S(t)\phi+\int_0^t S(t-s)f(u(s))\, \dee  s.\label{eq:VoC}
\end{equation}
We now make precise our solution concepts (see  e.g., \cite[Section 15]{QS19}), setting  $Q_T= \R^n\times (0,T)$.
\begin{definition}
\bi
\item[]
\item[(i)] We say that $u$ is a \emph{classical $L^1$-solution} of (\ref{nhe}) on $[0,T)$ if  $u$ satisfies (\ref{nhe}) in the classical sense in $Q_T$,
$u\in C\left([0,T),L^{1}(\R^n  )\right)\cap L^{\infty}_{\loc}\left((0,T),L^{\infty}(\R^n  )\right)$  and $u(0)=\phi$.
\item[(ii)] We say that (\ref{nhe}) is \emph{well-posed in $L^{1}(\R^n  )$} if for all $\phi\in L^{1}(\R^n  )$ there exists  $T=T(\phi )>0$ and a unique  classical $L^1$-solution of (\ref{nhe}) on $[0,T)$.
\ei
\label{def:wellposed}
\end{definition}

\begin{definition}
Let $T>0$. We say that a measurable, finite almost everywhere (a.e.) function $w\colon Q_T\to\R$ is an integral supersolution (respectively subsolution) of (\ref{nhe}) on $Q_T$ if $w$ satisfies ${\mathscr F}(w;\phi)\le w$ (respectively ${\mathscr F}(w;\phi)\ge w$) a.e. in $Q_T$, with $\F$  as in (\ref{eq:VoC}).
If $w$ is both an integral supersolution and  an integral subsolution of (\ref{nhe}) on $Q_T$ then we say that $w$ is  an integral solution of (\ref{nhe}) on $Q_T$.
\label{def:super}
\end{definition}
In what follows we will often  write `supersolution'  instead of `integral supersolution' whenever the context is clear, and likewise for `subsolution'. 


We recall the following standard smoothing estimate for the heat semigroup when $1\le q\le r\le\infty$ and $\phi\in L^q(\R^n )$ (see  e.g., \cite[Lemma 7]{BC}
or \cite[Proposition 48.4]{QS19}):
\be\label{eq:smoothing}
  \left\|S(t)\phi \right\|_{r}\le t^{-\frac{n}{2}\left(\frac{1}{q} -\frac{1}{r} \right)}\|\phi\|_q,\qquad t>0.
 \ee

The following lemma is a minor adaptation of  \cite[Lemma 8]{BC} for bounded domains to the whole space. Here we make use of the fact that 
$S(t)$ is a $C_0$-semigroup  to offer a slightly different  proof to the one given there.

\begin{lemma} \label{lem:smoothing}
Let  $1\le q < r\le\infty$ and $\displaystyle{\A =\frac{n}{2}\left(\frac{1}{q} -\frac{1}{r} \right)}$. If $\phi\in L^q(\R^n )$ then
$$\lim_{t\to 0} t^\A\left\|S(t)\phi \right\|_{r}=0.$$
\end{lemma}
\begin{proof}  For any $\varepsilon>0$, 
\bs
t^{\A}\left\|S(t)\phi \right\|_{r}& \le & t^{\A}\left\|S(t)(\phi - S(\varepsilon)\phi) \right\|_{r}+t^{\A}\left\|S(\varepsilon)S(t)\phi \right\|_{r}\\
& \le & \left\|\phi - S(\varepsilon)\phi \right\|_{q}+t^\alpha\varepsilon^{-\alpha}\left\|S(t)\phi \right\|_{q}\\
& \le & \left\|\phi - S(\varepsilon)\phi \right\|_{q}+t^\alpha\varepsilon^{-\alpha}\left\|\phi \right\|_{q}
\es
so that
\bs
\lim_{t\to 0} t^{\A}\left\|S(t)\phi \right\|_{r} \le  \left\|\phi - S(\varepsilon)\phi \right\|_{q}.
\es
Letting $\varepsilon\to 0$ and using the strong continuity of $S(\cdot )$ for $q\in [1,\infty )$, yields the result.
\end{proof}
\begin{remark}
The two ingredients needed in the proof of Lemma~\ref{lem:smoothing}, namely  the smoothing estimate and the strong continuity of the semigroup, may be generalised to the setting of  fractional power spaces $X^\alpha$, $\alpha\geq 0$,  associated with a sectorial operator $A$, such that  $A:D(A) \subset X^0\to  X^0$,  where $X^0$ is a Banach space. Then $\{S(t)\}_{t\geq 0}$ is an analytic semigroup generated by $A$ which,  in particular, satisfies the bounds  
\[
\|S(t)\phi\|_{X^\alpha}\leq M\|\phi\|_{X^\alpha}, \qquad \|S(t)\phi\|_{X^\alpha}\leq M t^{-\alpha}\|\phi\|_{X^0}, \qquad t> 0,
\]
for some constant $M>0$. See  e.g., \cite{AC} and the references therein.
\end{remark}



\subsection{Local Existence}\label{sec:exist}

We introduce the following  monotonicity hypothesis for $f$:
\begin{flalign*}
\quad\text{ \bf (M)}\quad f\colon\R\to\R \text{ is locally Lipschitz continuous, non-decreasing and } f(0)=0.&&
\end{flalign*}
For $f$ satisfying {\bf  (M)} we may define the  non-decreasing function  $\ell \colon [0,\infty )\to [0,\infty )$ by
\be
\ell (s)=\sup_{ 0<|t|\le s}\frac{f(t)}{t} \quad (s>0),\qquad \ell (0)=0.
\label{eq:ell}
\ee
Observe that by  {\bf  (M)} and   taking $t=s$ in (\ref{eq:ell}), we have
\be
f(s)\le s\ell(s),\qquad   s\ge 0\label{eq:fellplus}
\ee
and (taking $t=-s$),
\be
f(s)\ge s\ell(-s) ,\qquad s\le 0.\label{eq:fellminus}
\ee
We now consider the issue of existence  of solutions of (\ref{nhe}), for which we introduce our second hypothesis:
\begin{flalign*}
\quad\text{\bf  (I1)}\quad  \displaystyle{\int_1^{\infty}s^{-p_F}\ell (s)\,\dee s<\infty ,}&&
\end{flalign*}
where $p_F=1+2/n$. Observe that {\bf  (I1)} is equivalent to
\be
\int_{0}^1\ell \left(\tau^{-n/2}\right)\,\dee \tau <\infty .\label{eq:altI}
\ee

\begin{theorem}(Uniform Existence.)
Suppose $f$ satisfies {\bf  (M)} and {\bf  (I1)} and  $B$ is any bounded subset of $L^{1}(\R^n  )$. Then there exists a  $T_B>0$ such that for all $\phi\in B$ there exist
  classical $L^1$-solutions $\loweru (t;\phi)$ and $\upperu(t;\phi)$ of (\ref{nhe}) with $\loweru (t;\phi)\le \upperu(t;\phi)$ on $[0,T_B)$.
Furthermore, for all $\phi\in B$,
\[\lim_{t\to 0}t^{n/2}\| \loweru (t;\phi)\|_{\infty}=\lim_{t\to 0}t^{n/2}\| \upperu(t;\phi)\|_{\infty}=0.\]
\label{thm:exist}
\end{theorem}
\begin{proof}  For non-negative $\phi$, local
 existence  was proved in \cite[Theorem 5.1 (ii)]{LRSV}  (see also \cite[Theorem 4.4]{LRSV}).  We adapt that proof  here  to obtain solutions for initial data of indefinite sign, on a  uniform interval of existence with respect to $B$. We will show  that there exists a $T_B>0$ such that for any $\phi\in B$ there is an integral supersolution $w(t)$ and  an integral subsolution $v(t)$ with $v\le w$. From this sub-supersolution pair we will then construct  iteratively  the solutions $\loweru $ and $\upperu$. 

Let
\bs
\phi^{-}=\min\{\phi,0\}\le 0, \qquad \phi^{+}=\max\{\phi,0\}\ge 0
\es
and for $t\ge 0$ set
\bs
v(t)=2 S(t)\phi^{-}\le 0, \qquad w(t)=2 S(t)\phi^{+}\ge 0.
\es
Clearly $v\le 0\le w$ by monotonicity of $S(t)$ and the ordering $\phi^{-}\le 0\le  \phi^{+}$.

As $B$ is bounded there is a constant  $K>0$ such that $\|\phi\|_1\le K$ for all $\phi\in B$.  Hence since $w\ge 0$,  $\ell $ is non-decreasing and recalling (\ref{eq:fellplus}),
\begin{align}
S(t-s) f(w(s))&=S(t-s) f(2 S(s)\phi^{+}) 
\le S(t-s)  \ell   \left(2S(s)\phi^{+}\right)   2 S(s)\phi^{+} \nonumber \\
&\le S(t-s) \ell   \left(\left\|2 S(s)\phi^{+}\right\|_{\infty}\right) 2 S(s)\phi^{+} 
\le \ell   \left(2 \|\phi^{+}\|_{1}s^{-n/2}\right)2 S(t)\phi^{+} \nonumber \\
&\le \ell   \left(2\|\phi\|_{1} s^{-n/2}\right) w(t)\leq \ell \left(2 Ks^{-n/2}\right) w(t)\nonumber
\end{align}
and so
\begin{align*}
{\mathscr F}(w;\phi)&=S(t)\phi+\int_0^t S(t-s) f(w(s))\,\dee s\\
&\le S(t)\phi^{+}+\left(\int_0^t \ell   \left(2K s^{-n/2}\right)\,\dee s\right) w(t)\\
&= \left(1/2 +(2 K)^{2 /n}\int_0^{t(2 K)^{-2/n}} \ell (\tau^{-n/2})\,\dee\tau\right) w(t).
\end{align*}
Consequently ${\mathscr F}(w;\phi)\le w$ whenever
\bs
w(t)\left(g(t)-1/2\right)\le  0,\label{ss3}
\es
where $g$ is given by
\be
g(t)=(2 K)^{2 /n}\int_0^{t(2 K)^{-2 /n}}\ell (\tau^{-n/2})\,\dee\tau .
\label{ss4}
\ee
Due to \eqref{eq:altI} we can choose $T_B>0$ such that $g(t)\leq \frac{1}{2}$ for all $t\in [0,T_B]$, i.e.  ${\mathscr F}(w;\phi)\le w$ on $ [0,T_B]$. Note in particular that the choice of $T_B$ depends only upon $K$, and hence $B$, but not upon $\phi$.

Now recall  (\ref{eq:fellminus}) and that $v\le 0$.  Setting $\hat{\phi}=-\phi^{-}\ge 0$ we obtain
\begin{align}
S(t-s) f(v(s))&=S(t-s) f(2 S(s)\phi^{-}) 
\ge S(t-s)  \ell   \left(2 S(s)\hat{\phi}\right)   2 S(s)\phi^{-} \nonumber \\
&\ge S(t-s) \ell   \left(\left\|2 S(s)\hat{\phi}\right\|_{\infty}\right) 2 S(s)\phi^{-} 
\ge \ell   \left(2 \|\hat{\phi}\|_{1}s^{-n/2}\right)2 S(t)\phi^{-} \nonumber \\
&\ge \ell   \left(2 \|\phi\|_{1}s^{-n/2}\right) v(t)
\ge \ell   \left(2 Ks^{-n/2}\right) v(t),\label{eq:red2}
\end{align}
so that
\begin{align*}
{\mathscr F}(v;\phi) &= S(t)\phi+\int_0^t S(t-s) f(v(s))\,\dee s\\
&\ge S(t)\phi^-+\left(\int_0^t \ell   \left(2Ks^{-n/2}\right)\,\dee s\right) v(t)\\
&= \left(1/2+\int_0^t \ell   \left(2 Ks^{-n/2}\right)\,\dee s\right) v(t)\\
&=\left(1/2+g(t)\right)v(t),
\end{align*}
recalling (\ref{ss4}).
Thus, ${\mathscr F}(v;\phi)\ge v$  provided that
\bs
v(t)\left(g(t)-1/2\right)\ge  0,\label{eq:lower}
\es
which evidently holds for all $t\in [0,T_B]$, as for $w$ above.

For notational convenience we now set $T=T_B$ in the remainder of the proof. For any $\phi\in B$ and  $t\in [0,T]$, one may then use $v$ and $w$ to construct  iteratively a  decreasing sequence $w_k(t;\phi )$ and  an increasing sequence $v_k(t;\phi )$  via the relations
\be
w_{k+1}(t;\phi )={\mathscr F}(w_k;\phi), \qquad w_0(t;\phi )=w(t)\label{eq:itupper}
\ee
and
\be
v_{k+1}(t;\phi )={\mathscr F}(v_k;\phi), \qquad v_0(t;\phi )=v(t).\label{eq:itlower}
\ee
Using the monotonicity of ${\mathscr F}$ (see, for example, \cite[Theorem 1]{RS2} or \cite[Theorem 1]{W81}) it is easy to verify by induction that for all $k$
\bs
v\le v_k\le v_{k+1}\le w_{k+1}\le w_k\le w.\label{eq:itorder}
\es
Clearly therefore, the  pointwise limits 
\bs
\lim_{k\to\infty}v_k(x,t;\phi)=:\loweru(x,t;\phi ),\qquad  \lim_{k\to\infty}w_k(x,t;\phi )=:\upperu (x,t;\phi )\label{eq:ptlims}
\es
exist and satisfy $v\le\loweru\le \upperu\le w$ a.e. in $Q_{T}$. Furthermore, $v,w\in L^1(Q_{T})$ since
\bs
\int_{Q_{T}} |w(x,t)|\, \dee x\dee t &=&\ITB\Is 2[S(t)\phi^+](x)\, \dee x\dee t
\le 2T\|\phi\|_1,\\
\int_{Q_{T}} |v(x,t)|\, \dee x\dee t &=&\ITB\Is -2[S(t)\phi^-](x)\, \dee x\dee t
\le 2T\|\phi\|_1.
\es
Clearly, both  $|\loweru| $ and  $|\upperu|$ are dominated by $\max\{ |v|,|w|\}$.  Hence by the dominated convergence theorem, $\loweru , \upperu\in L^1(Q_{T})$ and
\bs
\lim_{k\to\infty}\It\Is G(x-y,s)v_k(y,s)\, \dee y\dee s & = &\It\Is G(x-y,s)\loweru (y,s)\, \dee y\dee s
\es
and
\bs
\lim_{k\to\infty}\It\Is G(x-y,s)w_k(y,s)\, \dee y\dee s & = &\It\Is G(x-y,s) \upperu (y,s)\, \dee y\dee s.
\es
It follows that  $\loweru$ and $ \upperu$ are both integral solutions of  (\ref{nhe}), i.e. satisfy
\be
u(t)= S(t)\phi+\int_0^t S(t-s)f(u(s))\, \dee  s\label{eq:VoC1}
\ee
a.e. in $Q_{T}$. Furthermore, by the integral condition {\bf  (I1)} we  have that
\bs
\ITB\|f(v(s))\|_1\, \dee s & \le &\ITB \ell   \left(2\|\phi^{-}\|_{1}s^{-n/2}\right)\|2S(s)\phi^{-}\|_1 \, \dee s\\
&\le & 2\|\phi\|_1\ITB \ell   \left(2\|\phi\|_{1}s^{-n/2}\right) \, \dee s<\infty
\es
and
\bs
\ITB\|f(w(s))\|_1\, \dee s & \le &\ITB \ell   \left(2s^{-n/2}\|\phi^{+}\|_{1}\right)\|2S(s)\phi^{+}\|_1 \, \dee s\\
&\le & 2\|\phi\|_1\ITB \ell   \left(2\|\phi\|_{1}s^{-n/2}\right) \, \dee s<\infty .
\es
Hence $f( \upperu )\in L^1((0,T),L^1(\R^n ))$ since $f(v)\le f( \upperu)\le f(w)$. It follows that  $ \upperu\in C([0,T],L^1(\R^n ))$ by (\ref{eq:VoC1}).
Furthermore, since  $v,w\in L^\infty_{\rm loc}((0,T),L^\infty(\R^n ))$ we also  have $ \upperu\in L^\infty_{\rm loc}((0,T),L^\infty(\R^n ))$. As $f$ is Lipschitz  continuous, standard parabolic regularity theory now implies that $ \upperu\in C^{2,1}(Q_T)$ and is a classical solution of (\ref{nhe}).   Thus, $ \upperu$ is a classical $L^1$-solution of (\ref{nhe}) on $[0,T)$.

By exactly the same argument  we  also deduce that   $\loweru$ is a classical $L^1$-solution of (\ref{nhe}) on $[0,T)$, with $\loweru \le  \upperu$.

That $t^{n/2}\|  \upperu (t;\phi)\|_{\infty}\to 0$ and $t^{n/2}\| \loweru (t;\phi)\|_{\infty}\to 0$ as ${t\to 0}$  follows easily from the ordering
\bs
2S(t)\phi^-\le \loweru (t;\phi)\le  \upperu(t;\phi)\le 2S(t)\phi^+
\es
and Lemma~\ref{lem:smoothing} with $q=1$ and $r=\infty$.
\end{proof}

\begin{remark}
For notational convenience, in all that follows we simply write $T_\phi$ for $T_{\{ \phi\}}$ when $B=\{\phi\}$.
\label{rem:T}
\end{remark}


\subsection{Uniqueness}\label{sec:unique}

The proof of Theorem~\ref{thm:exist} provides a particular method for constructing ordered  classical $L^1$-solutions $\loweru$ and  $ \upperu$ of (\ref{nhe}) via a monotone iteration scheme. Under a stronger integral condition than {\bf  (I1)}, involving  the modulus of continuity of $f$, we will show that $\upperu $  is unique among \emph {all} classical $L^1$-solutions of (\ref{nhe}). 

For $f$ satisfying {\bf  (M)}, define the non-decreasing function  $L\colon [0,\infty )\to [0,\infty )$ by
\be
L(s)=\sup_{\substack{ |u|,|v|\le s, \\ u\neq v}}\frac{f(u)-f(v)}{u-v}\quad (s>0),\qquad L (0)=0.
\label{eq:L}
\ee
We introduce the  integral condition 
\begin{flalign*}
\quad\text{\bf  (I2)}\quad  \int_1^{\infty}s^{-p_F}L (s)\,\dee s<\infty.&&
\end{flalign*}

\begin{remark}
Observe that $L(s)\ge\ell (s)$ and so  {\bf  (I2)} implies {\bf  (I1)}. We also note  that {\bf  (I2)} is equivalent to
\be
\int_{0}^1L \left(\tau^{-n/2}\right)\,\dee \tau <\infty .\label{eq:altI2}
\ee
\label{rem:I2I1}
\end{remark}

\begin{proposition}
Let $\phi,\psi\in L^1(\R^n  )$, $B=\{ \phi,\psi\}$ and suppose that  $f$ satisfies {\bf  (M)} and  {\bf  (I2)}. Let $\upperu (t;\cdot )$ and $T_B$ be as in Theorem~\ref{thm:exist}.
 Then there exists $\tau =\tau (\phi,\psi) >0$ and a continuous, non-negative function $q(t)$ satisfying $q(t)\to 0$ as $t\to 0$, such that for all $t\in (0,\tau ]$
\bs
\left\|\upperu (t;\phi)-\upperu (t;\psi)\right\|_1+t^{n/2}\left\|\upperu (t;\phi)-\upperu (t;\psi)\right\|_{\infty} \le 2\left\|\phi-\psi\right\|_1 \e^{q(t)}.
\es
\label{prop:cd}
\end{proposition}
\begin{proof}  To simplify notation  let   $u(t)=\upperu (t;\phi)$ and $v(t)=\upperu (t;\psi)$. By Theorem~\ref{thm:exist} there exists 
$\tau =\tau (\phi,\psi) \in (0,T_B)$ such  that
 \be
 \left\|u(t)\right\|_{\infty}\le t^{-n/2} \quad\text{and}\quad  \left\|v(t)\right\|_{\infty}\le t^{-n/2}\label{eq:uv}
 \ee
for all $t\in (0,\tau ]$.  For such $t$,
\begin{align}
\|u(t)-v(t)\|_1 & \le   \|S(t)(\phi-\psi)\|_1+ \It \left\|S(t-s)\left(f(u(s))-f(v(s))\right)\right\|_1\,\dee s\nonumber\\
&\le   \|\phi-\psi\|_1 +\It \left\|\frac{f(u(s))-f(v(s))}{u(s)-v(s)}\right\|_{\infty}\left\|S(t-s)\left(u(s)-v(s)\right)\right\|_1\,\dee s\nonumber\\
& \le  \|\phi-\psi\|_1 + \It L \left( s^{-n/2}\right)\left\|u(s)-v(s)\right\|_1\,\dee s\label{eq:L1}
\end{align}
by (\ref{eq:L}-\ref{eq:uv}). Next, we have
\begin{align}
\|u(t)-v(t)\|_{\infty}  &\le   \|S(t)(\phi-\psi)\|_{\infty}+\int_0^{t} \left\|S(t-s)\left(f(u(s))-f(v(s))\right)\right\|_{\infty}\,\dee s\nonumber\\
& \le  t^{-n/2}\|\phi-\psi\|_1+\int_{t/2}^t \left\|{f(u(s))-f(v(s))}\right\|_{\infty}\,\dee s\nonumber\\
&  +\int_0^{t/2} \left\|\frac{f(u(s))-f(v(s))}{u(s)-v(s)}\right\|_{\infty}\left\|S(t-s)(u(s)-v(s))\right\|_{\infty}\,\dee s\nonumber\\
& \le  t^{-n/2}\|\phi-\psi\|_1+ \int_{t/2}^t L \left(s^{-n/2} \right) \| {u(s)-v(s)}\|_{\infty}\,\dee s \nonumber\\
&   + \int_0^{t/2} L \left(s^{-n/2} \right) (t-s)^{-n/2}\| {u(s)-v(s)}\|_{1}\,\dee s\nonumber\\
& \le  t^{-n/2}\|\phi-\psi\|_1+(t/2)^{-n/2}\int_{t/2}^t L \left( s^{-n/2}\right) s^{n/2}\| {u(s)-v(s)}\|_{\infty}\,\dee s\nonumber\\
&  +(t/2)^{-n/2}\int_0^{t/2} L \left( s^{-n/2}\right)\| {u(s)-v(s)}\|_{1}\,\dee s\nonumber\\
& \le  t^{-n/2}\|\phi-\psi\|_1+(t/2)^{-n/2}\int_{0}^t L \left( s^{-n/2}\right) s^{n/2}\| {u(s)-v(s)}\|_{\infty}\,\dee s\nonumber\\
&   +(t/2)^{-n/2}\int_0^{t} L \left( s^{-n/2}\right)\| {u(s)-v(s)}\|_{1}\,\dee s.\label{eq:Linf}
\end{align}
Combining (\ref{eq:L1}-\ref{eq:Linf}) we obtain, for all $t\in (0,\tau ]$,
\begin{align}
 \|u(t)-v(t)\|_1& +t^{n/2}\|u(t)-v(t)\|_{\infty}   \le 2\|\phi-\psi\|_1 \nonumber\\
&+\int_0^{t} k_nL \left( s^{-n/2}\right)\left( \| {u(s)-v(s)}\|_{1}+ s^{n/2} \| {u(s)-v(s)}\|_{\infty}\right)\,\dee s,\label{eq:gron}
\end{align}
where $k_n=\left(1+2^{n/2}\right)$. 

Now define $y(t)$ on $[0,\tau ]$ by $y(0)=\|\phi-\psi\|_1$ and
\bs
y (t)=\|u(t)-v(t)\|_1 +t^{n/2}\|u(t)-v(t)\|_{\infty}, \quad t\in (0,\tau ].
 \es
 By assumption $u,v\in C\left([0,\tau ],L^{1}(\R^n  )\right)$ and are both classical solutions for $t>0$. By standard parabolic regularity results we have that
 $u,v\in  C\left((0,\tau ),L^{\infty}(\R^n  )\right)$,
  so that $y$ is continuous on $(0,\tau ]$. By Theorem~\ref{thm:exist} it  follows that $y$ is continuous on $[0,\tau ]$.
 Hence by (\ref{eq:gron}), {\bf  (I2)} and  the singular Gronwall inequality (see  e.g.,  \cite[Ch.XII, Theorem 4]{MPF}), it follows  that
\bs
\|u(t)-v(t)\|_1 +t^{n/2}\|u(t)-v(t)\|_{\infty}  & \le & 2\|\phi-\psi\|_1 \e^{q(t)}
\es
for all $t\in (0,\tau ]$, where
\bs
q(t)=k_n\int_0^t L \left( s^{-n/2}\right)\,\dee s.
\es
Clearly  $q$ is continuous with $q(t)\to 0$ as $t\to 0$ by {\bf  (I2)} (recalling (\ref{eq:altI2})), and the proof is complete.
\end{proof}

We are now in a position to establish our main uniqueness result.

\begin{theorem}(Uniqueness.)
Let $\phi\in L^1(\R^n  )$ and suppose that  $f$ satisfies {\bf  (M)} and  {\bf  (I2)}. Let $\upperu (t;\phi )$ and $T_{\phi}$ be as in Theorem~\ref{thm:exist}.
Then for any $T\le T_{\phi}$, $\upperu (t;\phi)$ is  the unique classical $L^1$-solution of (\ref{nhe}) on   $[0,T)$.
\label{thm:unique}
\end{theorem}
\begin{proof}  Let $u (t):=\upperu (t;\phi)$ for $t\in [0,T_{\phi})$ and $T\le T_{\phi}$ be arbitrary. Suppose  there  exists a   classical $L^1$-solution 
$v$ on $[0,T)$  with $v(0)=\phi$.
 By classical $L^{\infty}$-theory, it is clear that  if there exists a $\tau \in (0,T)$ such that $u (t)=v(t)$ on $(0,\tau ]$, then
  $u (t)=v(t)$ on $[0,T)$, i.e. uniqueness for sufficiently small times implies uniqueness on $[0,T)$.

Fix $\tau \in (0,T)$ and  let $B=v\left([0,\tau ]\right)$. Since $v\in C\left([0,\tau ],L^{1}(\R^n  )\right)$,  $B$
 is a  bounded  subset of $L^1(\R^n  )$.  By Theorem~\ref{thm:exist} there exists a $T_B\in (0,T_\phi ]$ such that   $\upperu (t;\psi )$ is well-defined for all $\psi\in B$ and $t\in [0,T_B)$. Hence for all $t\in [0,T_B)$ we may define $U_B(t)\colon B\to L^{1}(\R^n  )$  by $U_B(t)\psi =\upperu (t;\psi )$  and deduce from
 Proposition~\ref{prop:cd}  that $U_B(t)$ is continuous (with respect to the induced $L^1$-norm). Again
 by classical $L^{\infty}$-uniqueness theory we have that $v(t+s)=U_B(t)v(s)$ for all $t\ge 0$ sufficiently small and $s>0$ sufficiently small.
 Letting $s\to 0$, and using the continuity of $v:[0,\tau ]\to L^{1}(\R^n  )$ and $U_B(t)\colon B\to L^{1}(\R^n  )$, we therefore obtain $v(t)=U_B(t)v(0)=U_B(t)\phi=u (t)$ for all sufficiently small $t> 0$, as required.
\end{proof}

By Theorem~\ref{thm:unique}  we may  now dispense with the overbar (underbar) notation and simply write $u$ for $\upperu$ (and $\loweru$) 
whenever $\phi\in L^1(\R^n  )$  and  $f$ satisfies {\bf  (M)} and  {\bf  (I2)}.


\subsection{Continuous Dependence  and Comparison}

As a consequence of Theorem~\ref{thm:exist}  and Theorem~\ref{thm:unique} we may now  deduce the existence  of a maximally continued solution and obtain the usual blow-up alternative on its  maximal interval of existence.

\begin{theorem}(Maximal Solution.) Let $\phi\in L^1(\R^n  )$ and suppose that  $f$ satisfies {\bf  (M)} and  {\bf  (I2)}. Let $u (t;\phi )$ denote the unique classical $L^1$-solution on $[0, T_\phi)$. Then there exists $T_{\rm max}(\phi)\ge T_{\phi}$ such that:
\bi
\item[(a)]  $ u(t;\phi )$ can be continued (in a unique way) to a classical $L^1$-solution on $[0, T_{\rm max}(\phi))$;
\item[(b)] if $T_{\rm max}(\phi)< \infty$ then $u (t;\phi )$ cannot be continued to a classical $L^1$-solution on $[0, \tau )$
for any $\tau  > T_{\rm max}(\phi)$;
\item[(c)]  if $T_{\rm max}(\phi)< \infty$ then $\|u (t;\phi )\|_{1}\to\infty$ as $t\to T_{\rm max}(\phi)$.
\ei
\label{thm:maxsoln}
\end{theorem}

\begin{proof}
Follows in the standard way (see \cite[Proposition~16.1]{QS19} for example) by  the uniform existence result of Theorem~\ref{thm:exist} and
uniqueness of  Theorem~\ref{thm:unique}.
\end{proof}

Part (a) of the following theorem  is an immediate consequence of  Proposition~\ref{prop:cd} and Theorem~\ref{thm:unique}.  We choose to state it explicitly in order to summarize our results more clearly but especially because  Proposition~\ref{prop:cd} refers only to  solutions obtained via the monotone iteration scheme, as opposed to {\em any}  classical $L^1$-solution.

\begin{theorem} Let $\phi ,\psi\in L^1(\R^n  )$ and suppose that  $f$ satisfies {\bf  (M)} and  {\bf  (I2)}. Let   $u (t;\phi)$ and $u(t;\psi)$ be  as in Theorem~\ref{thm:maxsoln}.
\bi
\item[(a)] (Continuous Dependence.) There exists $\tau =\tau (\phi,\psi) >0$ and a continuous, non-negative function $q(t)$ satisfying $q(t)\to 0$ as $t\to 0$, such that for all $t\in (0,\tau ]$,
\bs
\left\|u (t;\phi)-u (t;\psi)\right\|_1+t^{n/2}\left\|u (t;\phi)-u (t;\psi)\right\|_{\infty} \le 2\left\|\phi-\psi\right\|_1 \e^{q(t)}.
\es
\item[(b)] (Comparison.) If  $\phi\le \psi$ then $u (t;\phi)\le u (t;\psi)$ on $[0,T)$, for all
$T\le \min\{T_{\rm{max}}(\phi),T_{\rm{max}}(\psi)\}$. In particular, if $\phi\ge 0$ then  $u(t;\phi)\ge 0$ on $[0,T_{\rm{max}}(\phi))$.
\ei
\label{thm:comp}
\end{theorem}
\begin{proof}
For part (b) we recall from Theorem~\ref{thm:exist} and Theorem~\ref{thm:unique} that there exists a  
 $T_\phi\le  T_{\rm{max}}(\phi)$ such that  the  classical $L^1$-solution $u(t;\phi)$ is obtained 
 as the pointwise limit of the iteration scheme
\bs
w_{k+1}(t;\phi )={\mathscr F}(w_k;\phi),\qquad w_0(t;\phi )=2 S(t)\phi^+
\es
for $t\in [0,T_\phi )$.

Similarly we obtain the   classical $L^1$-solution $u (t;\psi)$ on $[0,T_\psi)$ via
\bs
w_{k+1}(t;\psi )={\mathscr F}(w_k;\psi),\qquad w_0(t;\psi )=2 S(t)\psi^+.
\es
Let $T_0=\min\{T_\phi,T_\psi\}$. Clearly
\bs
w_0(t;\phi )=2 S(t)\phi^+ \le 2 S(t)\psi^+ =w_0(t;\psi ),
 \es
and since the  operator ${\mathscr F}$ is order-preserving in both its arguments (due to  monotonicity of $f$ and $S(t)$), it follows easily by induction that $w_k(t;\phi )\le w_k(t;\psi )$  for all $k$. Letting $k\to\infty$ we therefore have $u(t;\phi)\le u(t;\psi)$ on $[0,T_0)$.  The comparison result on  $[0,T)$ for any $T\le\max\{T_{\rm {max}}(\phi),T_{\rm {max}}(\psi )\}$ now follows by standard $L^\infty$-comparison results for classical solutions since $f$ is locally Lipschitz continuous.

The final part of (b) follows by comparison and the fact that $u(t;0)=0$ for all $t\ge 0$ by uniqueness.
\end{proof}


\subsection{Global Solutions}\label{sec:global}

Next we  establish some sufficient conditions for the global continuation of solutions with small intial data.
 The following integral condition plays the key  r\^ole:
\begin{flalign*}
\quad\text{\bf  (I3)} \quad \int_{0}^1s^{-p_F}\ell (s)\, \dee  s <\infty .&&
\end{flalign*}

\begin{theorem}
Suppose $\phi\in L^1(\R^n  )\cap L^\infty(\R^n  )$ and {\bf  (M)} holds.  Let  $u_c(t;\phi)$ denote  the unique classical $L^\infty$-solution
 of (\ref{nhe}) with maximal (in  $L^\infty$) interval of existence  $[0,T_{\rm max}(\phi))$. If   {\bf  (I3)} holds, then for any $A>1$
  there exists  $\delta = \delta (A)> 0$   such that for all  $\phi$ satisfying
$\|\phi\|_1+\|\phi\|_\infty \le \delta $ we have $T_{\rm max}(\phi)=\infty$ and
\be
AS(t)\phi^{-}\le u_c(t;\phi)\le AS(t)\phi^{+}\label{eq:globalorder}
\ee
for all $t\ge 0$.    Consequently $\|u_c(t;\phi)\|_\infty\le At^{-n/2}\|\phi\|_1$ for all $t>0$.
\label{thm:global}
\end{theorem}
\begin{proof}  First, we will show that for fixed $A>1$ and suitably small $\delta >0$, $v\coloneqq AS(t)\phi^-$ and $w\coloneqq AS(t)\phi^+$ are    an integral sub-supersolution pair for (\ref{nhe}) for all $t\ge 0$. Similarly to Theorem~\ref{thm:exist} we will then be able to deduce the existence of a pair of
classical $L^\infty$-solutions $\loweru$ and  $ \upperu$ satisfying
\be
v\le\loweru\le \upperu\le w.
\label{eq:infsubsup}
\ee
By   uniqueness of classical $L^\infty$-solutions (see e.g.,  \cite[Section 51, Appendix E]{QS19}) we may then conclude that $\loweru = \upperu=u_c$ and hence $v\le u_c(t;\phi)\le w$, yielding (\ref{eq:globalorder}). The $L^\infty$ bound for $u_c$ then follows by   $L^1$-$L^\infty$ smoothing via Lemma~\ref{lem:smoothing}. It is therefore sufficient
to prove the existence of solutions $\loweru$ and  $\upperu$ satisfying (\ref{eq:infsubsup}).


Let $\delta \le 1$ and choose  $\tau >0$ such that $1-A+A\ell (A)\tau <0$. By (\ref{eq:smoothing}) with $q=r=\infty$ we have $\|S(t)\phi\|_\infty\le \|\phi\|_\infty\le \delta \le 1$ for all $t\ge 0$. In particular, for all $t\in [0,\tau ]$ we have
\bs
{\mathscr F}(w;\phi )-w & = & S(t)\phi+\int_0^t S(t-s) f(w(s))\,\dee s-w\\
&\le &S(t)\phi^{+}+\int_0^t S(t-s) \left[\ell (w(s))w(s)\right]\,\dee s-w\\
&=& (1-A)S(t)\phi^{+}+\int_0^t S(t-s)\left[\ell \left(AS(s)\phi^{+}\right)AS(s)\phi^{+}\right]\,\dee s\\
&\le & (1-A)S(t)\phi^{+}+\int_0^t S(t-s)\left[\ell \left(\|AS(s)\phi^{+}\|_{\infty}\right)AS(s)\phi^{+}\right]\,\dee s\\
&\le & (1-A)S(t)\phi^{+}+\int_0^t S(t-s)\left[\ell (A)AS(s)\phi^{+}\right]\,\dee s\\
&=& (1-A)S(t)\phi^{+}+A\ell (A)\int_0^t S(t)\phi^{+}\,\dee s\\
&\le &\left(1-A+A\ell (A)\tau \right)S(t)\phi^{+}\le 0.
\es
For  $t> \tau $ we proceed as above, again making  use of   $L^1$-$L^\infty$ smoothing:
\bs
&&{\mathscr F}(w;\phi )-w= S(t)\phi+\int_0^\tau  S(t-s) f(w(s))\,\dee s+\int_\tau ^t S(t-s) f(w(s))\,\dee s-w\\
&\le&\left(1-A+A\ell (A)\tau\right)S(t)\phi^{+}+\int_\tau ^t S(t-s) \left[\ell (w(s))w(s)\right]\,\dee s\\
&\le& \left(1-A+A\ell (A)\tau\right)S(t)\phi^{+}+\int_\tau ^t S(t-s)\left[\ell \left(\|AS(s)\phi^{+}\|_{\infty}\right)AS(s)\phi^{+}\right]\,\dee s\\
&\le& \left(1-A+A\ell (A)\tau\right)S(t)\phi^{+}+\int_\tau ^t S(t-s)\left[\ell (A\delta s^{-n/2})AS(s)\phi^{+}\right]\,\dee s\\
&=& \left(1-A+A\ell (A)\tau\right)S(t)\phi^{+}+AS(t)\phi^{+}\int_\tau ^t \ell (A\delta  s^{-n/2})\,\dee s\\
&=& \left(1-A+A\ell (A)\tau+({2}/{n})A^{p_F}\delta^{2 /n}\int_{A\delta t^{-n/2}}^{A\delta \tau ^{-n/2}} z^{-p_F}\ell (z)\,\dee z\right)S(t)\phi^{+}\\
&\le& \left(1-A+A\ell (A)\tau +({2}/{n})A^{p_F}\delta^{2 /n}\int_{0}^{A\delta \tau ^{-n/2}} z^{-p_F}\ell (z)\,\dee z\right)S(t)\phi^{+}\\
&\le & 0
\es
for $\delta $ sufficiently small (and independently of $t$), by  {\bf  (I3)}.

In exactly the same way (and similar to that in the proof of Theorem~\ref{thm:exist} - see the calculations surrounding (\ref{eq:red2})) one may verify that
$v$ is an integral subsolution.
 We omit the repetitive details. Likewise one can again construct monotonic sequences $v_k$ and $w_k$ as in (\ref{eq:itupper}-\ref{eq:itlower}), converging pointwise to  limiting functions $\loweru$ and $\upperu$ respectively, with $v\le\loweru\le \upperu\le w$. Via the properties shared by $v$ and $w$ and the dominated convergence theorem, $\loweru$ and $\upperu$ both satisfy the variation of constants formula (\ref{eq:VoC}) almost everywhere in $Q_T$, for any $T>0$. Furthermore $\loweru$ and $\upperu$ are both essentially bounded in $Q_T$ and consequently they are classical solutions of (\ref{nhe}). This completes the proof.
\end{proof}

For initial data which are not necessarily bounded  we  also require $f$ to satisfy the integral condition  {\bf  (I2)}.

\begin{theorem}
Suppose $\phi\in L^1(\R^n)$ and {\bf  (M)} and  {\bf  (I2)}  hold.  Let  $u(t;\phi)$ denote  the unique classical $L^1$-solution
 of (\ref{nhe}) guaranteed by Theorem~\ref{thm:maxsoln}, with maximal  interval of existence $[0,T_{\rm max}(\phi))$.  If   {\bf  (I3)} holds, then for any $A>1$
 there exists  $\delta = \delta (A)> 0$   such that for all  $\phi$ satisfying
$\|\phi\|_1 \le \delta $ we have $T_{\rm max}(\phi)=\infty$ and
\bs
AS(t)\phi^{-}\le u(t;\phi)\le AS(t)\phi^{+}
\es
for all $t\ge 0$.    Consequently $\|u(t;\phi)\|_\infty\le At^{-n/2}\|\phi\|_1$ for all $t>0$.
\label{thm:globalL1}
\end{theorem}
\begin{proof}
It is easily verified that $v= AS(t)\phi^-$ and $w= AS(t)\phi^+$ are a pair of global integral  sub-supersolutions by simply taking $\tau =0$
 in the calculations in the proof of Theorem~\ref{thm:global}; for example,
 \bs
{\mathscr F}(w;\phi )-w \le  \left(1-A +({2}/{n})A^{p_F}\delta^{2 /n}\int_{0}^{\infty} z^{-p_F}\ell (z)\,\dee z\right)S(t)\phi^{+}
\le  0
\es
for $\delta $ sufficiently small  by  {\bf  (I2)} (see Remark~\ref{rem:I2I1}) and  {\bf  (I3)}. Similarly for $v$.

 The usual monotone iteration procedure then yields a pair of global classical $L^1$-solutions $\loweru$ and $\overline{u}$ satisfying
 $v\le\loweru\le \overline{u}\le w$ for all $t\ge 0$. By the  uniqueness result of Theorem~\ref{thm:unique} it follows that
  $\loweru =\overline{u}= u(t;\phi)$ on $[0,T_{\rm max}(\phi))$ and so $T_{\rm max}(\phi)=\infty$. The $L^\infty$ bound for $u$  follows
   once more by the ordering   $AS(t)\phi^-\le u\le  AS(t)\phi^+$ and  $L^1$-$L^\infty$ smoothing of the heat semigroup.
\end{proof}

\begin{example}
 Let $1<p<p_F<q$ and $f_+:[0,\infty )\to [0,\infty )$  be the non-convex, locally Lipschitz function
 \bs
 f_+(u)=\min\{u^{p},u^{q}\},
  \es
so that $f_+(u)=u^q$ for $0\le u\le 1$ and $f_+(u)= u^p$ for $u>1$.  Let  $f:\R\to\R$  be the odd extension of $f_+$. Then $f$ satisfies the hypotheses of both Theorem~\ref{thm:maxsoln} and Theorem~\ref{thm:globalL1}, with $L(s)=ps^{p-1}$ for $s$ large and $\ell (s)= s^{q-1}$ for $s$ small. This provides an example of a source term $f$ for which (\ref{nhe}) is both well-posed in $L^1(\R^n )$ {\em and}  possesses non-trivial global solutions, in contrast to the homogeneous power law case $f(u)=|u|^{p-1}u$, $p>1$.
\label{eg:minpq}
\end{example}




\section{Special Cases}

\subsection{Convex Source Terms}\label{sec:convex}

We now show that the `gap' between the integral condition {\bf  (I1)} for existence and the one  for uniqueness,  {\bf  (I2)},  vanishes
  when the source term $f$ is  odd and is convex on  $(0,\infty )$ .
\begin{lemma}
Assume {\bf  (M)} and {\bf  (I1)}  hold. If $f$ is  odd and is convex on  $(0,\infty )$, then $f$ satisfies  {\bf  (I2)}.
\label{lem:convexI1I2}
\end{lemma}
\begin{proof} 
Since $f$ is locally Lipschitz it is differentiable a.e., with derivative $f^\prime$ at all such points.
 The oddness of $f$ together with its convexity on $(0,\infty )$ then imply that  $\ell  (s)=f(s)/s$ and  $L(s)=f^\prime (s)$ a.e.
Hence,
\ba
\int_1^{\infty}s^{-p_F}L (s)\,\dee s &=& \int_{1}^\infty  s^{-p_F} f^\prime(s)\, \dee  s \nonumber\\
 &=& \left[s^{-p_F} f(s)\right]_{1}^{\infty} +p_F \int_{1}^\infty  s^{-p_F} \ell  (s)\, \dee  s.\label{eq:I2rhs}
\ea
The latter integral in (\ref{eq:I2rhs}) is finite by {\bf  (I1)}. Also,
\bs
\int_{s}^{2s}  t^{-p_F} \ell  (t)\, \dee  t\ge \ell  (s)\int_{s}^{2s}  t^{-p_F}\, \dee  t=C_ns^{-p_F} f(s)
\es
and so by  {\bf  (I1)} $s^{-p_F} f(s)\to 0$  as $s\to\infty$. Consequently the right hand side of (\ref{eq:I2rhs}) is finite and  {\bf  (I2)} holds.

\end{proof}

\begin{corollary}
Assume $f$ is odd, convex on $(0,\infty )$ and satisfies {\bf  (M)}. If 
$\int_1^{\infty} s^{-(2+2/n)}f (s)\,\dee s<\infty$  then the conclusions of Theorem~\ref{thm:maxsoln} and Theorem~\ref{thm:comp} hold.
\label{cor:convexI1I2}
\end{corollary}

\begin{example}
In the special case of the homogeneous power law $f(u)=|u|^{p-1}u$, $ p>1$,
Corollary~\ref{cor:convexI1I2} is applicable if and only if $p<p_F$. See for example \cite{BC,W79,W80}.
\label{eg:1}
\end{example}

One can also obtain a result like Corollary~\ref{cor:convexI1I2} without requiring $f$ to be odd, provided that one redefines {\bf  (I1)} and  {\bf  (I2)} accordingly.
For example, if $f$ is concave on $(-\infty ,0)$ and convex on  $(0,\infty )$, then $\ell  (s)=\max\{f(s)/s,-f(-s)/s\}$ and  $L(s)=\max\{f^\prime (s),-f^\prime (-s)\}$ a.e. One then replaces {\bf  (I1)} by the pair of integral conditions
\bs
\int_1^{\infty} s^{-(2+2/n)}f (s)\,\dee s<\infty \quad\text{and}\quad -\int_{-\infty}^{-1} (-s)^{-(2+2/n)}f(s)\,\dee s<\infty
\es
and analogously for  {\bf  (I2)}.


 Assuming  {\bf  (M)} holds  we have  shown that   {\bf  (I2)} is sufficient for the well-posedness of (\ref{nhe}) in $L^{1}(\R^n  )$, together with comparison and continuous dependence of solutions. It is natural therefore to ask whether  {\bf  (I2)} is {\it necessary} for  well-posedness. We have a partial result in this direction.

\begin{corollary}
Assume $f$ is odd, convex on $(0,\infty )$ and satisfies {\bf  (M)}. If  $f$ does not satisfy  {\bf  (I2)} then either (\ref{nhe}) is not well-posed   in $L^1(\R^n  )$ or the comparison principle fails.
\label{cor:I2nec}
\end{corollary}
\begin{proof} By Lemma~\ref{lem:convexI1I2}, if $f$  does not satisfy  {\bf  (I2)} then does not satisfy {\bf  (I1)}. Therefore, by \cite[Theorem 5.1]{LRSV} there exists a non-negative initial condition $\phi\in L^1(\R^n  )$ such that (\ref{nhe}) does not possess
a local {\em non-negative}  integral solution (and hence no non-negative classical $L^1$-solution). It follows that if (\ref{nhe}) is  well-posed
   in $L^1(\R^n  )$ then the corresponding unique classical $L^1$-solution $u(t;\phi )$ must be sign-changing on every small time interval $(0,T)$. Comparison with  $u(t;0 )=0$ consequently fails.
\end{proof}

Recalling Corollary~\ref{cor:convexI1I2} we may  combine the results of Theorem~\ref{thm:maxsoln} and Theorem~\ref{thm:globalL1} in the following special case.

\begin{corollary}
Suppose $f$ is  odd,  convex on $(0,\infty )$ and satisfies {\bf  (M)}.   If
\bs
\int_{0}^{\infty}\frac{f(s)}{s^{2+2 /n}}\, \dee  s <\infty ,\label{eq:easyffin}
\es
then (\ref{nhe}) is well-posed in $L^1(\R^n)$ and all classical $L^1$-solutions having sufficiently small initial data in $L^1(\R^n)$ are global in time,  decaying uniformly to zero like $O\left(t^{-n/2}\right)$ as $t\to\infty$.
\label{cor:locglob}
\end{corollary}


\subsection{Positive Solutions}\label{sec:pos}

Here we outline some consequences relating specifically to  {\em non-negative} solutions.
The results of previous sections are  easily paralleled by replacing $L^1(\R^n  )$  throughout by $L^1_+(\R^n  )$, the cone of non-negative functions in  $L^1(\R^n  )$. In particular  the definitions of solution and well-posedness in Definition~\ref{def:wellposed} are now made with respect to
 $L^1_+(\R^n  )$ rather than  $L^1(\R^n  )$, with $\phi\in L^1_+(\R^n  )$.

First we replace {\bf  (M)} by
\begin{flalign*}
\quad\text{\bf  (M)$_+$}\quad f\colon [0,\infty )\to [0,\infty ) \text{ is locally Lipschitz continuous, non-decreasing and } f(0)=0,&&
\end{flalign*}
and  (\ref{eq:ell}) by
\bs
\ell_+(s)=\sup_{0< t\le s}\frac{f(t)}{t} \quad (s>0),\qquad \ell_+ (0)=0.
\label{eq:ellplus}
\es
We then replace  $\ell$ by $\ell_+$ in the  integral condition  {\bf  (I1)}: 
\begin{flalign*}
\quad\text{\bf  (I1)$_+$}\quad\int_1^{\infty}s^{-p_F}\ell_+ (s)\,\dee s<\infty.&&
\end{flalign*}
The function $\loweru $ in Theorem~\ref{thm:exist} is then obtained by monotone iteration of the  integral subsolution $v=AS(t)\phi^-=0$, since now $\phi^-=0$.  Likewise we replace $L$ in (\ref{eq:L}) by
\bs
L_+(s)=\sup_{\substack{ 0<u,v\le s, \\ u\neq v}}\frac{f(u)-f(v)}{u-v} \quad (s>0),\qquad L_+ (0)=0
\label{eq:Lplus}
\es
and $L$ by $L_+$ in  {\bf  (I2)}:
\begin{flalign*}
\quad\text{\bf  (I2)$_+$}  \quad \int_1^{\infty}s^{-p_F}L_+ (s)\,\dee s<\infty.&&
\end{flalign*}
One readily deduces that if $f$ satisfies {\bf  (M)}$_+$ and {\bf  (I2)}$_+$ then for every $\phi\in L^1_+(\R^n  )$ there exists a 
$T_{\rm max}(\phi)> 0$ and a unique, non-negative classical $L^1$-solution $u (t;\phi )$ on $[0, T_{\rm max}(\phi))$.  Furthermore, if
$T_{\rm max}(\phi)< \infty$ then $\|u (t;\phi )\|_{1}\to\infty$ as $t\to T_{\rm max}(\phi)$. Similarly one obtains continuous dependence and comparison of non-negative solutions. This yields the analogues of Theorem~\ref{thm:maxsoln} and Theorem~\ref{thm:comp} in
$L^1_+(\R^n  )$.


The counterparts of Theorems~\ref{thm:global} and \ref{thm:globalL1} regarding   non-negative global solutions satisfying $0\le u(t;\phi)\le AS(t)\phi$ also follow in exactly the same way, on replacing    {\bf  (I3)}  with 
\begin{flalign*}
\quad\text{\bf  (I3)$_+$}  \quad \int_0^{1}s^{-p_F}\ell_+ (s)\,\dee s<\infty.&&
\end{flalign*}
In fact  the  r\^ole of {\bf  (I3)$_+$} is known to be important in determining whether positive classical solutions decay to zero or whether they
 `grow-up' \cite{KST}, and possibly blow-up.


\begin{example}
Let $f\colon [0,\infty )\to [0,\infty )$ be given by $f(0)=0$ and $f(u)=u^{p_F}g(u)$ for $u>0$, where
\bs
g(u)=\left\lbrace 
\begin{array}{ll}
\left[\ln(1/u)\right]^{-\gamma}, & 0<u<a,\\
g_0(u), & a\le u\le b,\\
\left[\ln{(\e+u)}\right]^{-\beta} & u>b,
\end{array}\right.
\es
$\gamma , \beta >1$ are fixed and $a,b>0$ will be chosen below. This example  combines those of  \cite[Section 4.4]{LRSV} on local existence of $L^1$-solutions and  \cite[Example 5.1]{KST} on global existence of classical solutions, in order to illustrate some of our  results. 

Choosing $a$ sufficiently small, $b$ sufficiently large and $g_0$ as a monotonic interpolant, we can ensure that $f$ satisfies  {\bf  (M)}$_+$ and moreover that $\ell_+(u)=f(u)/u$ on $(0,a)$ and $L_+(u)=f^\prime (u)$ for $u>b$. The choice of  $\beta >1$ and  $\gamma >1$  then ensure that 
  {\bf  (I2)}$_+$ and  {\bf  (I3)}$_+$ hold, respectively. Consequently (\ref{nhe}) is well-posed in $L^1_+(\R^n)$ with continuous dependence upon initial conditions and the comparison principle also assured. Furthermore, solutions with sufficiently small initial data in $L^1_+(\R^n)$ exist globally in time. 

We emphasise that the  results available in  \cite{BC,W79,W80} are not sufficiently sharp to be able to deduce the kind of well-posedness results that we obatin here for source terms such as this. 
\label{eg:2}
\end{example}

Similarly to Lemma~\ref{lem:convexI1I2} it again follows that if $f$  satisfies  {\bf  (M)}$_+$ and {\bf  (I1)}$_+$ and  is convex, then $f$  satisfies
    {\bf  (I2)}$_+$.  We therefore have the following interesting result:

\begin{corollary}
Suppose $f$ satisfies  {\bf  (M)}$_+$ and is convex. Then (\ref{nhe}) is well-posed in   $L^1_+(\R^n  )$ if and only if
 ${\int_1^{\infty} s^{-(2+2/n)}f (s)\,\dee s<\infty}$. Moreover, if  ${\int_1^{\infty} s^{-(2+2/n)}f (s)\,\dee s<\infty}$ then continuous dependence and comparison also hold in the sense of Theorem~\ref{thm:comp}.
\label{cor:summary}
\end{corollary}

\begin{proof}
If ${\int_1^{\infty} s^{-(2+2/n)}f (s)\,\dee s=\infty}$ then
by \cite[Theorem 5.1]{LRSV} there exists a  $\phi\in L^1_+(\R^n  )$ such that (\ref{nhe}) does not possess
a local {\em non-negative}  integral solution (and hence no non-negative classical $L^1$-solution). Consequently (\ref{nhe}) is not well-posed in
  $L^1_+(\R^n  )$.

Conversely, if 
${\int_1^{\infty} s^{-(2+2/n)}f (s)\,\dee s<\infty}$, then by Lemma~\ref{lem:convexI1I2}, Theorem~\ref{thm:maxsoln}, Theorem~\ref{thm:comp}   and  the discussion above, (\ref{nhe}) is well-posed in   $L^1_+(\R^n  )$ and enjoys the continuous dependence and comparison properties stated. 
\end{proof}


\section{Concluding remarks}\label{sec:conc}

We have established new results on the local well-posedness and global continuation of classical $L^1$-solutions of semilinear heat equations, extending those of \cite{BC, LRSV, W79,W80,W81} under less restrictive growth conditions on the source term $f$. Furthermore, we have also obtained   continuous dependence and comparison results in this more general setting. Here we discuss several extensions to our work which seem to us to be readily achievable. For expositional reasons  we have chosen not to present the details here; instead  we outline the necessary steps.

We have derived results only for  the Cauchy problem on $\R^n$. However, our results also hold in $L^1(\Om )$ for  bounded domains $\Omega$ with homogeneous Dirichlet or Neumann boundary conditions. The proofs require only minor (but frequent) modifications, along the lines of those in  \cite{LRSV},
and following the sub-supersolution methods used here with appropriate changes to the heat semigroup to incorporate the boundary conditions. In fact, by following the same argument as in \cite[Remark 7.2]{BC}  one can obtain uniqueness of  classical $L^1$-solutions in the larger class
 $$ C\left([0,T],L^{1}(\Omega  )\right)\cap L^{\infty}_{\loc}\left((0,T),L^{p_F}(\Omega  )\right)$$
 under the same hypotheses as Theorem~\ref{thm:unique}.

 Our well-posedness  results in  $L^1(\R^n  )$ also  carry through with minor modification if one replaces the Laplacian operator by the fractional Laplacian and consider instead the problem
 \bs\label{fracnhe}
u_t=-(-\Lap )^{\B/2} u+f(u),\qquad u(0)=\phi\in L^1(\R^n  )
\es
for $\B\in (0,2]$. The r\^ole of the Fujita-type exponent $p_F$ is then replaced by $p_F(\B )=1+\B /n$. Local and global existence results
 then follow by comparison with $AS_\B (t)\phi^{\pm}$ ($A>1$), where $S_\B (t)$ denotes the semigroup generated by
 the fractional Laplacian.  An appropriate  integral sub-supersolution existence theorem is easily obtained by adapting  those in \cite{RS2} and
  suitable monotonicity and smoothing properties are also available.
 Some work along these lines for local existence/non-existence of non-negative solutions can be found in \cite{Li}. One may then adapt our methods here to obtain the analogous uniqueness, continuous dependence and comparison results.





 \section*{Acknowledgements}
  This work was partially supported under the Research in Pairs scheme   at the Centre International de Rencontres Math\'ematiques, Luminy. 
 MS was also partially supported by NCN grant    2017/26/D/ST1/00614. The authors would also like to thank the anonymous reviewers for their helpful suggestions to improve the final version of this manuscript.

\bibliographystyle{model1num-names}
\bibliography{<your-bib-database>}

\begin{thebibliography}{00}



\bibitem{AC}
J.~Arrieta and A.~Carvalho.
\newblock Abstract parabolic problems
with critical nonlinearities and applications
to Navier-Stokes and heat equations.
\newblock {  Trans. Amer. Math. Soc.} 352 (2000), 285-310.



\bibitem{BP}
P.~Baras and M.~Pierre.
\newblock Crit\`ere d'existence de solutions positives pour des
              \'equations semi-lin\'eaires non monotones.
\newblock {  Ann. Inst. H. Poincar\'e Anal. Non Lin\'eaire}, 2 (1985), 185--212.


\bibitem{BC}
H.~Br\'{e}zis and T.~Cazenave.
\newblock A nonlinear heat equation with singular initial data.
\newblock {  J. Anal. Math.}, 68 (1996), 277--304.


\bibitem{CZ}
C.~Celik and Z.~Zhou.
\newblock No local {$L^1$} solution for a nonlinear heat equation.
\newblock {  Comm. Partial Differential Equations}, 28 (2003), 1807--1831.

\bibitem{FKWY}
M. ~Fila, J.R.~King, M.~Winkler and E.~Yanagida.
\newblock Grow-up rate of solutions of a semilinear parabolic equation with a critical exponent.
\newblock Adv. Differential Equations, 12 (2007), 1--26.

\bibitem{FHV}
 S.~Filippas, M.A.~Herrero and J.J.L.~Vel\'{a}zquez.
  \newblock Fast blow-up mechanisms for sign-changing solutions of a semilinear parabolic equation with critical nonlinearity. 
\newblock Proc. R. Soc. Lond. Ser. A Math. Phys. Eng. Sci., 456 (2000), 2957--2982.
 

\bibitem{FI}
Y.~Fujishima and N.~Ioku.
\newblock Existence and nonexistence of solutions for the heat equation with a superlinear source term.
\newblock {  J. Math. Pures Appl.}, 118 (2018), 128--158.

\bibitem{Fuj66}
H.~Fujita.
\newblock On the blowing up of solutions of the Cauchy problem for $u_{t}=\Delta u+u^{1+\alpha}$.
\newblock {  J. Sci. Univ. Tokyo}, Sect. I,  13 (1966), 109--124.






\bibitem{HW}
A.~Haraux and F.B.~Weissler.
\newblock Non-uniqueness for a semilinear initial value problem.
\newblock {  Indiana Univ. Math. J.}, 31 (1982), 167--189.

\bibitem{Hay}
K.~Hayakawa.
\newblock On nonexistence of global solutions of some semilinear parabolic differential equations.
\newblock {  Proc. Japan Acad. Ser. A Math. Sci.}, 49 (1973), 503--505.


\bibitem{Kaplan}
S.~Kaplan.
\newblock On the growth of solutions of quasi-linear parabolic equations.
\newblock {  Comm. Pure Appl. Math.} 16 (1963), 305--330.


\bibitem{KST}
K.~Kobayashi, T.~Sirao and H.~Tanaka.
\newblock On the growing up problem for semilinear heat equations.
\newblock {  J. Math. Soc. Japan},  29 (3) (1977),  407--425.



\bibitem{LRSV}
R.~Laister, J.C.~Robinson, M.~Sier{\.z}{\polhk{e}}ga and A.~Vidal-L\'{o}pez.
\newblock  A complete characterisation of local existence for semilinear parabolic equations in Lebesgue spaces.
\newblock {  Ann. Inst. H. Poincar\'{e} Anal. Non Lin\'{e}aire}, 33 (6) (2016),   1519--1538.




\bibitem{Li}
K.~Li.
\newblock A characteristic of local existence for nonlinear fractional
heat equations in Lebesgue spaces.
\newblock {\it Comput. Math. Appl.} 73 (2017), 653--665.

\bibitem{MT}
J.~Matos and E.~Terraneo.
\newblock Nonuniqueness for a critical nonlinear heat equation with any initial data.
\newblock {  Nonlinear Analysis}, 55 (2003), 927--936.


\bibitem{MPF}
D.S.~Mitrinovi\'c, J.~Pe\v{c}ari\'c and A.M.~Fink.
\newblock {  Inequalities Involving Functions and Their Integrals and Derivatives}.
\newblock Mathematics and its Applications. Kluwer Academic Publishers, Dordrecht, 1991.



\bibitem{NS}
W-M.~Ni and P.~Sacks.
\newblock Singular behavior in nonlinear parabolic equations.
\newblock {  Trans. Amer. Math. Soc.}, 287 (1985), 657--671.


\bibitem{PY}
P.~Pol\'{a}\v{c}ik and E.~Yanagida.
\newblock On bounded and unbounded global solutions of a supercritical semilinear heat equation.
\newblock Math. Ann., 327 (2003), 745--771.

\bibitem{QS19}
P.~Quittner and P.~Souplet.
\newblock {  Superlinear Parabolic Problems. Blow-up, Global Existence and
  Steady States}  (2nd Edition).
\newblock Birkh\"auser Advanced Texts, Basel, 2019.



\bibitem{RS2}
J.C.~Robinson and M.~Sier{\.z}{\polhk{e}}ga.
\newblock Supersolutions for a class of semilinear heat equations.
\newblock {  Rev. Mat. Complut.}, 26 (2013), 341--360.




\bibitem{Sug}
S.~Sugitani.
\newblock On nonexistence of global solutions for some nonlinear integral equations.
\newblock {  Osaka J. Math.}, 12 (1975), 45--51.


\bibitem{W79}
F.B.~Weissler.
\newblock Semilinear evolution equations in {B}anach spaces.
\newblock {  J. Funct. Anal.}, 32 (1979), 277--296.


\bibitem{W80}
F.B.~Weissler.
\newblock Local existence and nonexistence for semilinear parabolic equations
  in {$L^{p}$}.
\newblock {  Indiana Univ. Math. J.}, 29 (1980), 79--102.


\bibitem{W81}
F.B.~Weissler.
\newblock Existence and nonexistence of global solutions for a semilinear heat equation.
\newblock {  Israel J. Math.}, 38 (1981), 29--40.

\end{thebibliography}

\end{document}